\documentclass[12pt]{amsart}
\usepackage{latexsym,fancyhdr,amssymb,color,amsmath,amsthm,graphicx,listings,comment}
\usepackage[section]{placeins}
 \newtheorem{propo}{Proposition}  
\let\paragraph\subsection

\title{Cartan's Magic Formula for Simplicial Complexes}
\author{Oliver Knill}
\date{11/25/2018}
\address{Department of Mathematics \\ Harvard University \\ Cambridge, MA, 02138 }
\subjclass{05E-xx,68R-xx,51P-xx}

\begin{document}
\maketitle

\begin{abstract}
\'Elie Cartan's magic formula $L_X = i_X d + d i_X = (d+i_X)^2=D_X^2$ relates
the exterior derivative $d$, an interior derivative $i_X$ and its Lie derivative $L_X$. 
We use this formula to define a finite dimensional vector 
space $\mathcal{X}$ of vector fields $X$ on a finite abstract simplicial complex $G$. 
The space $\mathcal{X}$ has a Lie algebra structure satisfying 
$L_{[X,Y]} = L_X L_Y - L_Y L_X$ as in the continuum. Any such vector field $X$ defines a 
coordinate change on the finite dimensional vector space
$l^2(G)$ which play the role of translations along the vector field.
If $i_X^2=0$, the relation $L_X=D_X^2$ with $D_X=i_X+d$ mirrors the Hodge factorization $L=D^2$, where
$D=d+d^*$ we can see $f_t =  - L_X f$ defining the flow of $X$ as the analogue of
the heat equation $f_t = - L f$ and view the Newton type equations $f_{tt} = -L_X f$ as the analogue of the
wave equation $f_{tt} = -L f$. Similarly as the wave equation is solved by 
$\psi(t)=e^{i Dt} \psi(0)$ with $\psi(t)=f(t)-i D^{-1} f_t(t)$, 
also any second order differential equation  $f_{tt} = -L_X f$ is solved by 
$\psi(t) = e^{i D_X t} \psi(0)$ in $l^2(G,\mathcal{C})$. If $X$ 
is supported on odd forms, the factorization property $L_X = D_X^2$ extends to 
the Lie algebra and $i_{[X,Y]}$ remains an inner derivative. 
If the kernel of $L_X: \Lambda^p \to \Lambda^p$ has dimension
$b_p(X)$, then the general Euler-Poincar\'e formula $\chi(G) = \sum_k (-1)^k b_k(X)$ holds for every
parameter field $X$. Extreme cases are $i_X=d^*$, where $b_k$ are the usual Betti numbers and $X=0$, 
where $b_k=f_k(G)$ are the components of the $f$-vector of the simplicial complex $G$. 
We also note that the McKean-Singer super-symmetry extends from $L$ to Lie derivatives. (It 
also holds for $L_X$ on Riemannian manifolds but appears to have been unnoticed there so far):
the non-zero spectrum of $L_X$ on even forms is the same than the non-zero spectrum of 
$L_X$ on odd forms. We also can make a deformation
$D_X' = [B_X,D_X]$ of $D_X=d+i_X + b_X, B_X=d_X-d_X^*+i b_X$ which produces a in general 
non-isospectral deformation of the
exterior derivative $d$ governed by the vector field $X$ featuring inflationary initial size
decay for $d$ typical for such systems, leading to an expansion of space. 
\end{abstract} 

\section{Introduction}

\paragraph{}
When formulating physics in a discrete geometric frame work, one is challenged by the absence 
of a continuous diffeomorphism group. What is the analogue of a vector field on a
finite abstract simplicial complex $G$? Discrete theory approaches like 
\cite{Forman1999,MarsdenDesbrun,RomeroSergeraert} have put forward some notions. 
What about a combinatorial discrete frame work \cite{AmazingWorld}?

\paragraph{}
In order to get a continuum motion, one always can just look at paths in the 
unitary group on the Hilbert space $l^2(G)$ 
like for example isospectral Lax deformations $D'=[B,D]$ of the exterior derivative $d$ 
defining $D=d+d^*$ \cite{IsospectralDirac,IsospectralDirac2}. But we would like to have a definition of 
vector field which is formally identical to the classical case and which agrees with 
the classical case if the differential complex comes from a Riemannian manifold. 

\paragraph{}
As in the continuum, like on a Riemannian manifold \cite{AMR,FrankelGeometryPhysics}, we would like to 
see vector fields related to $1$-forms, possibly moderated by a Riemannian metric,
The presence of an exterior derivative $d$
then produces potential fields $F=dV$ which then can be used to generate dynamics like
$x''=-\nabla V(x)$. In that case, the Riemannian structure allows to transfer the $1$-form $dV$ 
into a vector field $\nabla V$. In this note, we look at a
vector field notion in the discrete which works on any finite abstract simplicial complex $G$, 
a finite set of non-empty sets $x$ invariant under the process of taking non-empty finite subsets.

\paragraph{}
We see that there is a finite dimensional Lie algebra of vector fields $X$ 
for which the Lie derivative $L_X$ defines a coordinate change 
which commutes with exterior differentiation $d$. The coordinate changes allow to have 
a basic general covariance principle. The motion can be extended to a Hamiltonian frame work so
that we could look also at analogues the Kepler problem, where a mass point moves 
in a central field given by a potential associated to the geometry. Given a vector field $X$,
then the solution $\psi(t) = e^{i D_X t} \psi(0)$ of the second order and so 
a Newton type equation $f_{tt} = -L_X f$ with $L_X=D_X^2$  and 
$\psi(0)=f(0)-i D_X^{-1} f_t(0)$ resembles then the solution $e^{i D t} \psi(0)$ of the wave equation 
$f_{tt} = -L f$ with $L=D^2$. It is the Cartan magic formula \cite{CartanRiemannianGeometry}, which is 
now part of any differential topology book like \cite{AMR}) 
which produces the analogy between the wave equation with Hodge Laplacian $L=d d^* + d^* d = (d+d^*)^2=D^2$ 
and the Lie derivative $L_X = d i_X + i_X d = (d+i_X)^2 = D_X^2$. The {\bf Cartan formula} 
can therefore be seen as a key to port notions of ordinary differential equations on manifolds to 
discrete spaces like simplicial complexes. 

\paragraph{}
The Cartan formula has been used in the past in discrete frame works (it appears in \cite{MMPDTKMD}).
Usually, in the Newton case as well as in the wave case, one does not write the dynamics
using complex coordinates. Here in the discrete, it is convenient as the equations become just
Schr\"odinger equations giving paths in the unitary group. 
The wave equation case with Laplacian $L$ is the most
symmetric case, where the propagation of information happens in all directions or 
(if the momentum $f_t$ is chosen accordingly) allows to force the propagation into a specific direction. 
The analogue of ordinary differential equations are obtained when replacing $L$ with $L_X$, 
in which  $X$ is a deterministic field, which assigns to a simplex, just one smaller 
dimensional simplex.

\paragraph{}
The note draws also from insight gained in \cite{AnnieThesis,AnnieProject} and belongs 
to the theme of looking for finite dimensional analogues of partial differential equations in the continuum. 
The set-up for \cite{AnnieThesis} builds on work like \cite{Chapman} and is
an {\bf advection model} for a {\bf directed graph} $G$ is $u' = -Lu$, where
$L={\rm div} (V ({\rm grad}_0(u)))$ is the directed Laplacian on the graph. As in the usual (scalar) Laplacian 
$L=d^*d$ for undirected graphs which leads to the {\bf heat equation} $u' = -Lu$, the
{\bf advection Laplacian} uses difference operators:
the modification of the gradient ${\rm grad}_0$ which is the maximum of ${\rm grad}$ and $0$.
The divergence ${\rm div}=d^*$ as well as ${\rm grad}=d$ are defined by the usual
exterior derivative $d$ on the graph.
The {\bf consensus model} is the situation, where the graph $G$ is replaced by its
{\bf reverse graph} $G^T$, where all directions are reversed. One can therefore focus
on advection. A central part of \cite{AnnieThesis} relates this to Markov chains. If $L=D-A$, then
$M=A D^{-1}$ is a {\bf left stochastic matrix}, a Markov operator, which maps
probability vectors to probability vectors. The kernel of $L$ is related to the
fixed points of $M$. Assume $L D^{-1} u=0$, then $(D-A) D^{-1} u=0$ and $u=A D^{-1} u = Mu$.
Perron-Frobenius allows to study the structure of the equilibria which are given by the
kernel of $L$. This concludes the diversion into advection. 

\paragraph{}
What is the connection?
While the topic is related, we look here at differential equations on all differential
forms and not only on $0$-forms. Also, we don't yet really study the dynamics much and 
just establish the linear algebra set-up showing that 
there is an elementary way to define a Lie algebra of vector
fields in a discrete set-up. The affinity to the continuum is that the formalism is not only 
similar but identical to the continuum. Whatever is done works both for Riemannian manifolds
or finite abstract simplicial complexes or more generally for a differential complex. 
There are many open questions as seen at the end of this note:
integer-valued deterministic $X$ often produce integer eigenvalues of $L_X$ for example. We
would like to know when this is case appears. 

\paragraph{}
An other angle emerged while teaching the multi-variable Taylor theorem in \cite{Math22-2018}. 
Already the single variable Taylor theorem can be seen as the solution $f(t,x)=e^{D t} f(0)$ of the 
{\bf transport equation} $f_t = D f$ with $D=d/dx$, a partial differential equation. 
Because also $f(t,x)=f(x+t)$ solves this 
partial differential equation, we have $f(x+t)=e^{Dt} f(0)$ which becomes so the Taylor theorem, 
provided the initial function $f(0)$ is real analytic. 
For a multivariate function we can replace $L=D^2$ with a Lie derivative $L_X=D_X^2$ and in the case of a 
constant field $X=v$, a Taylor expansion $f(x+tv) = f(x) + df(x) tv + d^2f (x) (tv)^2/2 + ...$
(The multivariable Taylor theorem can be formulated conveniently using directional (Fr\'echet) 
derivatives which is how textbooks like \cite{Edwards1973,BlatterIII1974} treats the subject in higher
dimensions, avoiding tensor calculus.) 
As both $L$ and $L_X$ can be written as a square $L=D^2, L_X=D_X^2$, the analogy between the wave and Newton equation
has appeared. The frame work shows that a``diffeomorphism Lie group" exists in any geometric structure with 
an exterior derivative; Taylor links the vector field Lie algebra with translation. 

\section{From Cartan to d'Alembert} 

\paragraph{}
Given a smooth compact manifold $M$ or a simplicial complex $G$ with {\bf exterior derivative} 
$d: \Lambda^p \to \Lambda^{p+1}$, then every vector field $X$ defines an {\bf interior derivative} 
$i_X: \Lambda^{p} \to \Lambda^{p-1}$. The {\bf Cartan magic formula} writes the {\bf Lie derivative} 
$L_X$ as $L_X=d i_X + i_X d$. From the identities $d^2=0$ and $i_X^2=0$ follows that $L_X$ 
commutes with $d$. We know already from the continuum, that without the $d i_X$ part, the naive directional 
derivative $i_X d$ alone would not work, as it would be coordinate dependent. 
A $L_X$ commutes with $d$ it leads to a chain homotopy between the complexes before and after 
the coordinate transformation. Like the {\bf Hodge Laplacian} $L=d d^* + d^* d$, {\bf we can write
$L_X$ as a square}: define $D_X = d+i_X$ and $D=d+d^*$. Then, $L_X=D_X^2$ and $L=D^2$. 
The {\bf directional Dirac operator} $D_X$ has also an adjoint $D_X^*$ but it is different from $D_X$ 
in general. Despite the notation used, the directional Dirac operator is not the directional derivative
used in calculus. The operators $d,i_X,D_X$ and $L_X$ work on the linear space of all differential forms.

\paragraph{}
If $L=D^2$ is the Hodge Laplacian with Dirac operator $D=d+d^*$, then the {\bf wave equation} 
is $f_{tt} = -L f$. The {\bf directional wave equation} is the formal analogue
$f_{tt} = -L_X f$. Written in the d'Alembert form, it is $(\partial_{tt}+L_X) f=0$. 
As $L$ and $L_X$ are both squares of simpler operators $D$ and $D_X$, we can factor
$(\partial_t + i D_X) (\partial_t - i D_t) f =0$ or $(\partial_t + i D) (\partial_t - i D) f =0$. 
The solutions $e^{\pm i D_X}$ which with Euler's formula $e^{i Dt} = \cos(Dt) + i \sin(Dt)$ leads
to the explicit solution $f(t)=\cos(Dt) f(0) + \sin(Dt) D^{-1} f_t(0)$, where $D^{-1}$ is the 
pseudo-inverse is defined as $D^{-1} f_t^\perp(0)$ if $f_t^\perp(0)$ is in the orthogonal 
complement of the kernel of $D$. 

\paragraph{}
So far, the solutions of the equations were real-valued functions in the Hilbert space $H=l^2(G,\mathbb{R})$. 
If $G$ is a finite abstract simplicial complex, the Hilbert space is finite dimensional, and the frame work 
is part of linear algebra.
It is convenient to build the {\bf complex valued wave} $\psi(0) = f(0) - i D^{-1} f_t(0)$,
(where again $D^{-1}$ is the pseudo inverse) and get $\psi(t) = e^{i Dt} \psi(0)$. 
The dynamics is now the solution to the {\bf Schr\"odinger equation} $i\psi_t = -D \psi$, 
where $\psi(0)$ encodes the initial position $f(0)$ in its real part and the initial velocity $f_t(0)$ 
in its imaginary part. This works in the same way for $D_X$. In both cases, the wave equation for the real 
wave is equivalent to the Schr\"odinger equation for a complex wave. 
In the case of $0$-forms, we have $L=d^* d$ and $L_X=i_X d$ is the {\bf directional derivative} in the
direction $X$. We summarize: 

{\it Given a geometric space with an exterior derivative $d$, the second order 
real wave equation $f_{tt} = -L f$ is equivalent to the first order complex
Schr\"odinger equation $\psi' = i D \psi$, leading to a d'Alembert solution 
$\psi(t) = e^{i D t} \psi(0)$ which can then be computed using a Taylor expansion and for
which the real part of $\psi$ gives the wave solution $f(t)$. }

\section{The Lie algebra}

\paragraph{}
Let us assume now that we are in a finite dimensional geometric space with an exterior derivative 
$d: \Lambda^p \to \Lambda^{p+1}$ leading to a differential complex on a graded vector space 
$\Lambda=\oplus_{p=0}^{{\rm dim}(G)} \Lambda^p$. 
A {\bf vector field} is defined by a linear operator $i_X$ on $\Lambda$ which maps
$\Lambda^p$ to $\Lambda^{p-1}$ and has the property that $i_X^2 = 0$. We can define a Lie algebra
multiplication $Z=[X,Y]$ by first forming $L_X = i_X d + d i_X$ which is a map from  
$\Lambda^p$ to $\Lambda^p$ and then defining $Z$ through the inner derivative
$$   i_Z = i_{[X,Y]} = [ L_X,i_Y ] = L_X i_Y - i_Y L_X \; . $$

\paragraph{}
The field $Z$ can be read of from $i_Z$. We also have the {\bf Lie algebra relation} 
\begin{eqnarray*}
 L_Z  &=&  i_Z d + d i_Z = L_X i_Y d            - i_Y L_X d  + d L_X i_Y  - d i_Y          L_X \\
      &=&                  L_X (L_Y-d i_Y)      - i_Y L_X d  + L_X d i_Y  - (L_Y  - i_Y d) L_X \\
      &=&                  L_X L_Y - L_Y L_X   \; . 
\end{eqnarray*}

\paragraph{}
Now, if $i_X i_Y = i_Y i_X = 0$, then
$L_X i_Y - i_Y L_X = i_X L_Y - L_Y i_X$ because
inter changing $X$ and $Y$ produces a change of sign of $L_Z$.
As $L_Z = i_Z d + d i_Z$, also $i_Z$ changes sign
meaning $Z$ changes sign. So, $i_Z = L_X i_Y - i_Y L_X = - (L_Y i_X - i_X L_Y)$.  

\paragraph{}
These elementary matrix identities prove 
the following proposition which applies to any so derived Lie algebra of fields $X$ with $i_X^2=0$. 

\begin{propo}
Every vector field $X$ defines
an operator $D_X = d+i_X$ which has as a square a Lie derivative $L_X=D_X^2$. The set of $L_X$
define a Lie algebra with $L_{[X,Y]}=L_X L_Y - L_Y L_X$. If $X$ is in supported on odd forms,
then $i_X^2=0$ and $L_X=D_X^2$ holds in the entire Lie algebra. 
\end{propo}

\paragraph{}
Even so $L_X$ is not self-adjoint in general, it plays the 
role of a Laplacian. In the case when $L_X$ is not diagonalizable, we can not form the pseudo inverse
when writing the dynamics in the complex but we can assume that the initial velocity is 
in the image of $D_X$. With this definition, also the adjoint operator $d^*$ defines
a vector field. We can still see $d^*=i_X$ for some field $X$. It belongs to the class of vector fields 
$X$ for which the eigenvalues of $D_X$ are real. 

\paragraph{}
We could look at a subclass of ``deterministic fields", which have the property that for a $p$-form $f$, 
the $(p-1)$-form $i_X f$ is supported on a single sub-simplex $y$ of $x$, if $f$ is supported on a single 
simplex $x$. This would correspond to classical vector fields close to the discrete Morse theory frame-work
\cite{Forman1999}.  If $X$ is supported on one-dimensional simplices, this is close to the 
discrete Morse theory set-up. Unlike in the continuum, these ``deterministic fields"
are not invariant under addition, nor under the Lie algebra multiplication $[X,Y]$. 
When taking the commutator of $L_X$ and $L_Y$ for such fields, 
then the corresponding inner derivative $i_{[X,Y]}$ connects simplices which are not 
directly connected. 

\paragraph{}
If the simplicial complex is one-dimensional, or if $X$ is restrained to $p$-forms with odd $p$ (or even if we like)
then $i_Z=i_{[X,Y]}$ satisfies again $i_Z^2=0$ and the factorization $L_Z=(i_Z+d)^2= D_Z^2$ holds in the
entire Lie algebra. In general, the new interior derivative is only nilpotent, because
$i_Z^{1+{\rm dim}(G)} =0$. The Mathematica procedures below allow to support $X$ onto any subset of forms but
we mostly use the case when $X$ is supported on odd-dimensional forms. 

\paragraph{}
Let us briefly look at the $1$-dimensional (single-variable) classical case $M=\mathbb{R}$, where 
$L_X=i_X d$ is what we understand to be the usual derivative $d/dx$.
Technically, the exterior derivative $d$ produces from a $0$-form $f \in \Lambda^0$  
a $1$-form $df dx \in \Lambda^1$ which is in a different vector space than $f$. But for the constant 
vector field $X=1$, the combination $i_X d$ produces again an element in $\Lambda^0$. 
Since $i_X^*=0$ on $0$-forms, we have $i_X d = i_X d + d i_X = L_X$. Now 
$e^{L_X t} f = \sum_{n=0}^{\infty} (d/dx)^n f(x)X^n t^n/n!$ is by the {\bf Taylor formula} equal 
to $f(x+t X)$, illustrating that the derivative $d/dx$ is the generator of the translation. 
The flow $\phi_X f = f(x+t X)$ is a solution of a {\bf transport equation} and not the wave equation. 
To get an analogue of the later, we need a second derivative in time and so a symplectic or complex 
structure. It is first a bit puzzling to see the Lie derivative $L_X$ as a second order operator.
But the Cartan formula shows that also in one dimensions, $L_X = i_X d = (d + i_X)^2 = D_X^2$ is second
order. In calculus, we usually think of the derivative $d$ as a map on a space of {\bf scalar functions} and
not as a map from $0$-forms to $1$-forms. The inner derivative $i_X$ which brings us back to $0$-forms
is silently assumed in calculus. This identification of 0-forms and 1-forms can not be done in the 
discrete because the dimension $v_1=|E|$ of $1$-forms is different from the dimension
$v_0=|V|$ of $0$-forms on a graph $G=(V,E)$. Still, the general frame work applies and the 
wave equation $f_{tt} = - L_X f$ can be written as $(\partial_t - i D_X)  (\partial_t + i D_X)=0$
with $D_X = d + i_X$. 

\paragraph{}
We would like to point out that the McKean-Singer symmetry 
\cite{McKeanSinger} which holds for simplicial complexes and the 
operator $L=D^2$ \cite{knillmckeansinger} remains valid also for $L_X=D_X^2$:

\begin{propo}[McKean-Singer symmetry]
The non-zero spectrum of $L_X$ on even differential forms $\Lambda^{even}$ 
is the same than the non-zero spectrum of $L_X$ on odd
differential forms $\Lambda^{odd}$.
\end{propo}

\begin{proof}
The proof is the same than in the continuum \cite{Cycon} 
or used in \cite{knillmckeansinger}: the operator $D_X$ which 
exchanges $\Lambda^{even}$ and $\Lambda^{odd}$ gives
a translation between the eigenvectors belonging to non-zero
eigenvalues. The discrepancy between the kernels on odd and
even forms is by definition the Euler characteristic. 
\end{proof} 

\paragraph{}
Also the nonlinear {\bf Lax type deformation} 
\cite{IsospectralDirac,IsospectralDirac2} 
of the Dirac operator generalizes from $D$ to $D_X$. These Lax equations are
$$   D_X' = [B_X,D_X]  $$ 
where $D_X=d+i_X + b_X, B_X=d_X-d_X^*+i b_X$. 
Unlike for $i_X=d^*$, where $L=L_X$ is the Hodge Laplacian, the deformation 
is now not isospectral in general and therefor not expected to be integrable. 
The expansion rate in different part of space or differential forms happens differently.
Still, these systems remain interesting non-linear differential equations 
and the corresponding $d(t)$ still satisfies $d(t)^2=0$ producing an exterior
derivative after deformation. As in the case of the wave equation, the deformed
$d(t)$ keeps the same cohomology.

\section{Physics}

\paragraph{}
Any mathematical theory with some quantum gravitational ambitions should be able to be powerful enough
to solve the Kepler problem effectively in any scale: in the large, it should lead to the 
classical Kepler problem, in the very large to relativistic motion in a Schwarzschild metric and 
in the very small to the quantum dynamics in the Hydrogen atom. No current theory passes this 
Kepler test: no theory can yet describe a point in the influence 
of a central field classically, relativistically and quantum mechanically, not just in principle 
or a perturbative patch work  but in an elegant manner, leading to quantitative results 
which match experiments in all three scales. It should be able to describe 
the motion of satellites or planets, also relativistically, predict the emission patters of 
gravitational waves emitted by a binary system or the structure of the Mendeleev table in the small. 

\paragraph{}
The general covariance principle in physics states that physical laws
are independent of the coordinate system. This means that the laws should not 
only be invariant under a finite dimensional symmetry group like Euclidean or Lorentz symmetry 
but they should be invariant under the diffeomorphism group of the manifold. 
Additionally, in case of fibre bundles, additional gauge symmetries might apply, but this
is also part of the general covariance principle. An example are the Maxwell equations $dF=0,
d^*F=j$ leading in the Coulomb gauge $d^* A=0$ to the Poisson equation $L A=j$. 
If we move into a new coordinate system, then the transported equations look the same. 
An other example are the Einstein field equations $G = e T$, relating the 
geometric Einstein tensor with the energy tensor $T$ using a proportionality factor $e$, the
Einstein constant. The covariance there is there the statement that $G$ and 
$T$ are tensors. How can one port the general covariance principle to the discrete? 
A naive request would be to look at laws only which are invariant after applying a deformation 
through a vector field. Since $L_X$ and $d$ commute, any law which only involves the
exterior derivative does this. Examples are the wave, the heat or the Schr\"odinger equation. 

\section{Examples}

\paragraph{}
The simplest case with a non-trivial Lie algebra is when 
$G=\{ \{1\},\{2\},\{1,2\}\}$ is the Whitney complex of the complete graph $K_2$. 
In that case, 
$$ d   = \left[ \begin{array}{ccc} 0 & 0 & 0 \\ 0 & 0 & 0 \\ -1 & 1 & 0 \\ \end{array} \right], 
   d^* = \left[ \begin{array}{ccc} 0 & 0 & -1 \\ 0 & 0 & 1 \\ 0 & 0 & 0 \\ \end{array} \right] \; . $$
The general inner derivative (vector field) has the form 
$$ i_X = \left[ \begin{array}{ccc} 0 & 0 & a \\ 0 & 0 & b \\ 0 & 0 & 0 \\ \end{array} \right] \; . $$
The general operator $D_X=d + i_X$ and $L_X=D_X^2= d i_X + i_X d$ then is 
$$ D_X = \left[ \begin{array}{ccc} 0 & 0 & a \\ 0 & 0 & b \\ -1 & 1 & 0 \\ \end{array} \right], 
   L_X = \left[ \begin{array}{ccc} -a & a & 0 \\ -b & b & 0 \\ 0 & 0 & b-a \\ \end{array} \right]  \; . $$
The eigenvalues of $L_X$ are $\{ 0, b-a, b-a \}$. Given an other vector field 
$$ i_Y=\left[ \begin{array}{ccc} 0 & 0 & u \\ 0 & 0 & v \\ 0 & 0 & 0 \\ \end{array} \right]  \; , $$
one can form $i_Z  = L_X i_Y-i_Y L_X = i_X L_Y - L_Y i_X$ which is 
$$ i_Z=\left[ \begin{array}{ccc} 0 & 0 & a v-b u \\ 0 & 0 & a v-b u \\ 0 & 0 & 0 \\ \end{array} \right] \; , $$
leading to 
$$ D_Z=\left[ \begin{array}{ccc} 0 & 0 & a v-b u \\ 0 & 0 & a v-b u \\ -1 & 1 & 0 \\ \end{array} \right], 
   L_Z=\left[ \begin{array}{ccc} b u-a v & a v-b u & 0 \\ b u-a v & a v-b u & 0 \\ 0 & 0 & 0 \\ \end{array} \right] \; .$$
This satisfies $L_Z = L_X L_Y - L_Y L_X$. The eigenvalues of $L_Z$ are all zero. Indeed $L_Z^2=0$. 

\paragraph{}
Here is the general case if $G=\{ \{1\},\{2\}, \{3\},\{1,2\},\{1,3\} \}$ is the Whitney 
complex of a linear graph of length $2$. 
$$ d = \left[    \begin{array}{ccccc}
                   0 & 0 & 0 & 0 & 0 \\
                   0 & 0 & 0 & 0 & 0 \\
                   0 & 0 & 0 & 0 & 0 \\
                   -1 & 1 & 0 & 0 & 0 \\
                   -1 & 0 & 1 & 0 & 0 \\
                  \end{array} \right], 
   d^* = \left[  \begin{array}{ccccc}
                   0 & 0 & 0 & -1 & -1 \\
                   0 & 0 & 0 & 1 & 0 \\
                   0 & 0 & 0 & 0 & 1 \\
                   0 & 0 & 0 & 0 & 0 \\
                   0 & 0 & 0 & 0 & 0 \\
                  \end{array} \right] \; . $$
With a vector fields 
$$   i_X = \left[
                  \begin{array}{ccccc}
                   0 & 0 & 0 & a_1 & a_2 \\
                   0 & 0 & 0 & a_3 & 0   \\
                   0 & 0 & 0 &  0  & a_6 \\
                   0 & 0 & 0 & 0 & 0 \\
                   0 & 0 & 0 & 0 & 0 \\
                  \end{array} \right],
     D_X = \left[
                  \begin{array}{ccccc}
                   0 & 0 & 0 & a_1 & a_2 \\
                   0 & 0 & 0 & a_3 &  0  \\
                   0 & 0 & 0 &  0  & a_6 \\
                  -1 & 1 & 0 & 0 & 0 \\
                  -1 & 0 & 1 & 0 & 0 \\
                  \end{array} \right] \; . $$
This leads to 
$$ L_X = \left[
                  \begin{array}{ccccc}
                   -a_1-a_2 & a_1 & a_2 & 0 & 0 \\
                   -a_3     & a_3 &  0  & 0 & 0 \\
                       -a_6 &   0 & a_6 & 0 & 0 \\
                   0 & 0 & 0 & a_3-a_1 &    -a_2 \\
                   0 & 0 & 0 &    -a_1 & a_6-a_2 \\
                  \end{array} \right] \; . $$
Given an other vector field
$$   i_Y = \left[
                  \begin{array}{ccccc}
                   0 & 0 & 0 & b_1 & b_2 \\
                   0 & 0 & 0 & b_3 &  0   \\
                   0 & 0 & 0 & 0   & b_6 \\
                   0 & 0 & 0 & 0 & 0 \\
                   0 & 0 & 0 & 0 & 0 \\
                  \end{array} \right],
     D_Y = \left[
                  \begin{array}{ccccc}
                   0 & 0 & 0 & b_1 & b_2 \\
                   0 & 0 & 0 & b_3 &  0  \\
                   0 & 0 & 0 &  0  & b_6 \\
                  -1 & 1 & 0 & 0 & 0 \\
                  -1 & 0 & 1 & 0 & 0 \\
                  \end{array} \right] \; , $$
we get 
\begin{tiny}
$$ i_Z = \left[ \begin{array}{ccccc}
 0 & 0 & 0 & -a_2 b_1-a_3 b_1+a_1 (b_2+b_3) & a_2 (b_1+b_6)-(a_1+a_6) b_2 \\
                   0 & 0 & 0 & a_1 b_3-a_3  b_1 & a_2 b_3-a_3 b_2 \\
                   0 & 0 & 0 & a_1 b_6-a_6 b_1 & a_2 b_6-a_6 b_2 \\
                   0 & 0 & 0 & 0 & 0 \\
                   0 & 0 & 0 & 0 & 0 \\
                  \end{array} \right] \; . $$
\end{tiny}
The eigenvalues of $L_X$ contain $0$ as well as the following two eigenvalues,
each with multiplicity two: \\ 
$\left(\pm \sqrt{a_1^2+2 a_1 (a_2-a_3+a_6)+(a_2+a_3-a_6)^2}-a_1-a_2+a_3+a_6\right)/2$. 
We see that real eigenvalues are quite common but that imaginary eigenvalues of 
$L_X$ can occur in the $4$-dimensional space of vector fields. The operator $L_Z$
does in general not have zero eigenvalues. They can even become complex. 
We also see that the commutator $L_Z$ now tunnels between places which were not
directly connected in $G$. 

\paragraph{}
For the complete complex $G=\{\{1\},\{2\},\{3\},\{1,2\},\{2,3\},\{1,3\},\{1,2,3\}\}$,
we can look at the general 
$$ D_X = d + iX = \left[ \begin{array}{ccccccc}
  0 & 0 & 0 & a & b & 0 & 0 \\
  0 & 0 & 0 & c & 0 & d & 0 \\
  0 & 0 & 0 & 0 & e & f & 0 \\
 -1 & 1 & 0 & 0 & 0 & 0 & g \\
 -1 & 0 & 1 & 0 & 0 & 0 & h \\
  0 & -1 & 1 & 0 & 0 & 0 & i \\
  0 & 0 & 0 & 1 & -1 & 1 & 0 \\
                 \end{array} \right] \; . $$
Now, given two general $D_X,D_Y$, we have $L_X=D_X^2, L_Y=D_Y^2$.
If $g=h=i=0$ then, $L_X i_Y-i_Y L_X = i_X L_Y-L_Y i_X$
and $i_Z=L_X i_Y-i_Y L_X$ has the property that $i_Z^2=0$. 

\paragraph{}
Let $G=\{\{1\},\{2\},\{3\},\{4\},\{1,2\},\{2,3\},\{3,4\},\{4,1\}\}$ 
be the complex of the cyclic graph $C_4$. The exterior derivative $d$ and an example of an 
interior derivative $i_X$ are:

\begin{tiny}
$$          d =  \left[
                 \begin{array}{cccccccc}
0 & 0 & 0 & 0 & 0 & 0 & 0 & 0 \\
 0 & 0 & 0 & 0 & 0 & 0 & 0 & 0 \\
 0 & 0 & 0 & 0 & 0 & 0 & 0 & 0 \\
 0 & 0 & 0 & 0 & 0 & 0 & 0 & 0 \\
 -1 & 1 & 0 & 0 & 0 & 0 & 0 & 0 \\
 0 & -1 & 1 & 0 & 0 & 0 & 0 & 0 \\
 0 & 0 & -1 & 1 & 0 & 0 & 0 & 0 \\
 1 & 0 & 0 & -1 & 0 & 0 & 0 & 0 \\
                 \end{array}
                 \right],
           i_X = \left[
                 \begin{array}{cccccccc}
 0 & 0 & 0 & 0 & -1 & 0 & 0 & 0 \\
 0 & 0 & 0 & 0 & 1 & 0 & 0 & 0 \\
 0 & 0 & 0 & 0 & 0 & 1 & 0 & 0 \\
 0 & 0 & 0 & 0 & 0 & 0 & 1 & 0 \\
 0 & 0 & 0 & 0 & 0 & 0 & 0 & 0 \\
 0 & 0 & 0 & 0 & 0 & 0 & 0 & 0 \\
 0 & 0 & 0 & 0 & 0 & 0 & 0 & 0 \\
 0 & 0 & 0 & 0 & 0 & 0 & 0 & 0 \\
                 \end{array}
                 \right] \; . $$
This leads to the Dirac operator $D=d+d^*$ and the directional Dirac operator
$D_X=d+i_X$
$$          D   = \left[
                  \begin{array}{cccccccc}
 0 & 0 & 0 & 0 & -1 & 0 & 0 & 1 \\
 0 & 0 & 0 & 0 & 1 & -1 & 0 & 0 \\
 0 & 0 & 0 & 0 & 0 & 1 & -1 & 0 \\
 0 & 0 & 0 & 0 & 0 & 0 & 1 & -1 \\
 -1 & 1 & 0 & 0 & 0 & 0 & 0 & 0 \\
 0 & -1 & 1 & 0 & 0 & 0 & 0 & 0 \\
 0 & 0 & -1 & 1 & 0 & 0 & 0 & 0 \\
 1 & 0 & 0 & -1 & 0 & 0 & 0 & 0 \\
                  \end{array}
                  \right],
            D_X = \left[
                  \begin{array}{cccccccc}
 0 & 0 & 0 & 0 & -1 & 0 & 0 & 0 \\
 0 & 0 & 0 & 0 & 1 & 0 & 0 & 0 \\
 0 & 0 & 0 & 0 & 0 & 1 & 0 & 0 \\
 0 & 0 & 0 & 0 & 0 & 0 & 1 & 0 \\
 -1 & 1 & 0 & 0 & 0 & 0 & 0 & 0 \\
 0 & -1 & 1 & 0 & 0 & 0 & 0 & 0 \\
 0 & 0 & -1 & 1 & 0 & 0 & 0 & 0 \\
 1 & 0 & 0 & -1 & 0 & 0 & 0 & 0 \\
                  \end{array}
                  \right]  \; . $$
The Hodge Laplacian $L=D^2$ and Lie derivative $L_X=D_X^2$ are 
$$  L = \left[
                  \begin{array}{cccccccc}
 2 & -1 & 0 & -1 & 0 & 0 & 0 & 0 \\
 -1 & 2 & -1 & 0 & 0 & 0 & 0 & 0 \\
 0 & -1 & 2 & -1 & 0 & 0 & 0 & 0 \\
 -1 & 0 & -1 & 2 & 0 & 0 & 0 & 0 \\
 0 & 0 & 0 & 0 & 2 & -1 & 0 & -1 \\
 0 & 0 & 0 & 0 & -1 & 2 & -1 & 0 \\
 0 & 0 & 0 & 0 & 0 & -1 & 2 & -1 \\
 0 & 0 & 0 & 0 & -1 & 0 & -1 & 2 \\
                  \end{array}
                  \right], 
    L_X =  \left[
                  \begin{array}{cccccccc}
 1 & -1 & 0 & 0 & 0 & 0 & 0 & 0 \\
 -1 & 1 & 0 & 0 & 0 & 0 & 0 & 0 \\
 0 & -1 & 1 & 0 & 0 & 0 & 0 & 0 \\
 0 & 0 & -1 & 1 & 0 & 0 & 0 & 0 \\
 0 & 0 & 0 & 0 & 2 & 0 & 0 & 0 \\
 0 & 0 & 0 & 0 & -1 & 1 & 0 & 0 \\
 0 & 0 & 0 & 0 & 0 & -1 & 1 & 0 \\
 0 & 0 & 0 & 0 & -1 & 0 & -1 & 0 \\
                  \end{array}
                  \right] \; . $$
\end{tiny}
The eigenvalues of $D$ are $\{-2, 2, -\sqrt{2}, -\sqrt{2}, \sqrt{2}, \sqrt{2}, 0, 0\}$, the eigenvalues
of $D_X$ are $\{ -\sqrt{2}, -\sqrt{2}, -1, -1, 1, 1, 0, 0 \}$. 
The eigenvalues of $L_X$ are $\{ 4, 4, 2, 2, 2, 2, 0, 0 \}$.

\begin{figure}[!htpb]
\scalebox{0.64}{\includegraphics{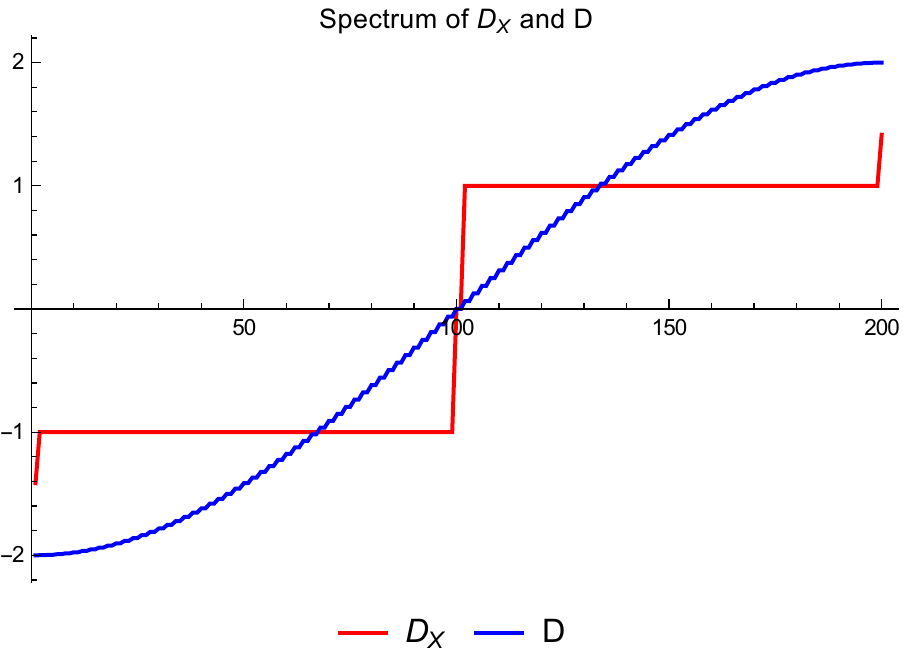}}
\scalebox{0.64}{\includegraphics{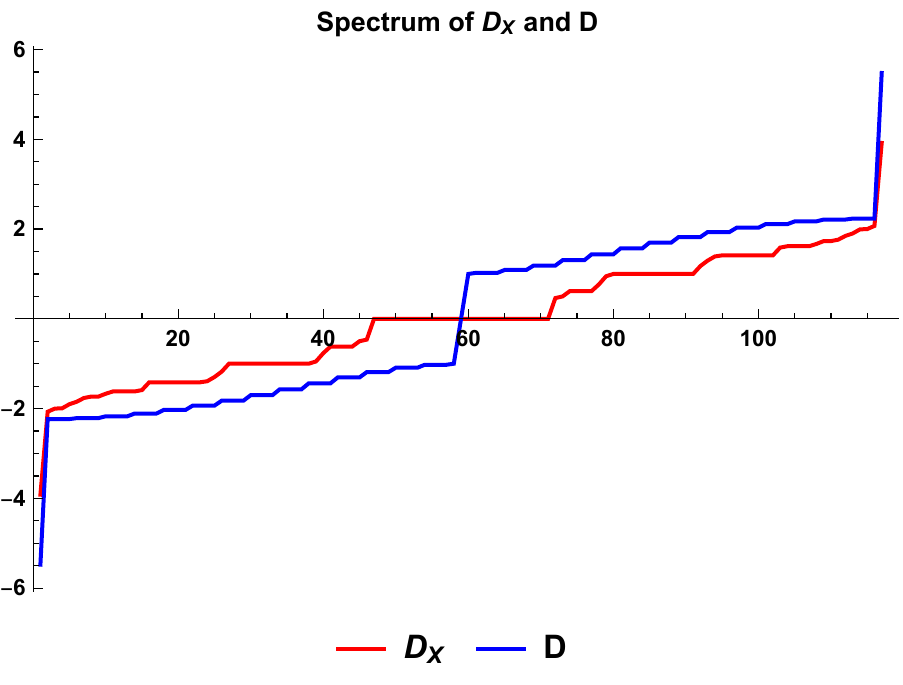}}
\label{example}
\caption{
The figures show the spectrum of $D$ and $D_X$ for a cycle graph spectrum and a wheel graph spectrum. 
In both cases, we start with $i_X=d^*$, then take away each entry $1$ or 
$-1$ with probability $1/2$. $D_X$ still has the $\sigma(D_X)=-\sigma(D_X)$ symmetry
from the Dirac operator $D$ but the energies of the particles are smaller as it is 
more difficult to travel. 
}
\end{figure}

\section{Mathematica procedures}

\paragraph{}
Here is the code which computes the Dirac operator $D$ and the
vector field analogue $D_X$ for any simplicial complex $G$. The 
code can be copy pasted when accessing the LaTeX source of this document
on the ArXiv. The first part computes the
matrices given in the above example. For the vector field, we chose
for $i_X$ just to take the first non-zero entry of $d^*$: 

\begin{tiny}
\lstset{language=Mathematica} \lstset{frameround=fttt}
\begin{lstlisting}[frame=single]
G={{1},{2},{3},{4},{1,2},{2,3},{3,4},{4,1}};  
n=Length[G];Dim=Map[Length,G]-1;f=Delete[BinCounts[Dim],1];
Orient[a_,b_]:=Module[{z,c,k=Length[a],l=Length[b]},
  If[SubsetQ[a,b] && (k==l+1),z=Complement[a,b][[1]];
  c=Prepend[b,z]; Signature[a]*Signature[c],0]];
d=Table[0,{n},{n}];d=Table[Orient[G[[i]],G[[j]]],{i,n},{j,n}];
dt=Transpose[d]; DD=d+dt; LL=DD.DD;
HX[x_]:=Block[{u=Flatten[Position[Abs[x],1]]},If[u=={},0,First[u]]];
iX=Table[0,{n},{n}];
Do[l=HX[dt[[k]]];If[l>0,iX[[k,l]]=dt[[k,l]]],{k,f[[1]]}];
DX=iX+d; LX = DX.DX;  
\end{lstlisting}
\end{tiny}

\paragraph{}
Here is the code to generate $D_X$ and $L_X$ for a generate random 
finite abstract simplicial complex $G$: 

\begin{tiny}
\lstset{language=Mathematica} \lstset{frameround=fttt}
\begin{lstlisting}[frame=single]
(* Generate a random simplicial complex                               *)
Generate[A_]:=Delete[Union[Sort[Flatten[Map[Subsets,A],1]]],1]
R[n_,m_]:=Module[{A={},X=Range[n],k},Do[k:=1+Random[Integer,n-1];
  A=Append[A,Union[RandomChoice[X,k]]],{m}];Generate[A]]; 
G=Sort[R[10,20]];

(* Computation of exterior derivative                                 *)
n=Length[G];Dim=Map[Length,G]-1;f=Delete[BinCounts[Dim],1];
Orient[a_,b_]:=Module[{z,c,k=Length[a],l=Length[b]},
  If[SubsetQ[a,b] && (k==l+1),z=Complement[a,b][[1]];
  c=Prepend[b,z]; Signature[a]*Signature[c],0]];
d=Table[0,{n},{n}];d=Table[Orient[G[[i]],G[[j]]],{i,n},{j,n}];
dt=Transpose[d]; DD=d+dt; LL=DD.DD;

(* Build interior derivatives iX  and iY                              *)
UseIntegers=False;
e={}; Do[If[Length[G[[k]]]==2,e=Append[e,k]],{k,n}]; 
BuildField[P_]:=Module[{X,ee,iX=Table[0,{n},{n}]},
X=Table[If[UseIntegers,Random[Integer,1],Random[]],{l,Length[e]}];
Do[ee=G[[e[[l]]]]; Do[If[SubsetQ[G[[k]],ee], 
  m=Position[G,Sort[Complement[G[[k]],Delete[ee,2]]]][[1,1]]; 
  iX[[m,k]]= If[MemberQ[P,Length[G[[m]]]],X[[l]],0]*
             Orient[G[[k]],G[[m]]]],{k,n}],{l,Length[e]}]; iX];

(* Build Laplacians LX,LY,LZ, plot spectrum of D and DX and matrices  *)
iX=BuildField[{1,3,5,7,9}]; iY=BuildField[{1,3,5,7,9}];
DX=iX+d; LX=DX.DX; DY=iY+d; LY=DY.DY; iZ1=LX.iY-iY.LX; iZ2=iX.LY-LY.iX; 
Print[iZ1==iZ2]; iZ=iZ1; LZ=Chop[iZ.d+d.iZ]; DZ=iZ+d;     
dx="\!\(\*SubscriptBox[\(D\), \(X\)]\)";
lx="\!\(\*SubscriptBox[\(L\), \(X\)]\)"; pl=PlotLabel;
GraphicsGrid[{{MatrixPlot[DX,pl->dx], MatrixPlot[LX,pl->lx]},
              {MatrixPlot[DD,pl->"D"], MatrixPlot[LL,pl->"L"]}}];
u1 = Sort[Chop[Eigenvalues[1.0 DX]]]; u2 = Sort[Eigenvalues[1.0 DD]];
u1=N[Round[u1*10^6]/10^6];            (* clear tiny imaginary parts  *)
S=ListPlot[{u1, u2},Joined ->True,PlotLegends ->Placed[{dx,"D"},Below],
 PlotRange -> All, PlotStyle -> {Red, Blue},
 PlotLabel -> "Spectrum of \!\(\*SubscriptBox[\(D\), \(X\)]\) and D"];

(* Compute Betti numbers, compare bosonic and fermionic part          *)
chi=Sum[-f[[k]](-1)^k,{k,Length[f]}]; f=Prepend[f,0]; m=Length[f]-1; 
U=Table[v=f[[k+1]];
   Table[u=Sum[f[[l]],{l,k}];LL[[u+i,u+j]],{i,v},{j,v}],{k,m}];
Cohomology = Map[NullSpace, U]; Betti = Map[Length, Cohomology]
chi1=Sum[-Betti[[k]](-1)^k,{k,Length[Betti]}];
EV=Map[Eigenvalues,U];
EVFermi=Table[EV[[2k]],{k,Floor[Length[EV]/2]}];
EVBoson=Table[EV[[2k-1]],{k,Floor[(Length[EV]+1)/2]}];
extract[u_]:=Module[{v=Flatten[u],w={}},
   Do[If[Abs[v[[k]]]>10^(-8),w=Append[w,v[[k]]]],{k,Length[v]}];Sort[w]];
extract[EVFermi]==extract[EVBoson] 

(* Now the same for LX                                               *)
U=Table[v=f[[k+1]];
   Table[u=Sum[f[[l]],{l,k}];LX[[u+i,u+j]],{i,v},{j,v}],{k,m}];
EV=Map[Eigenvalues,U];
EVFermi=Table[EV[[2k]],{k,Floor[Length[EV]/2]}];
EVBoson=Table[EV[[2k-1]],{k,Floor[(Length[EV]+1)/2]}];
extract[EVFermi]==extract[EVBoson] 
{EVFermi,EVBoson}
Total[Abs[N[extract[EVFermi] - extract[EVBoson]]]]

Cohomology = Map[NullSpace, U]; Betti = Map[Length, Cohomology];
chi2=Sum[-Betti[[k]](-1)^k,{k,Length[Betti]}];
{chi,chi1,chi2}
\end{lstlisting} \end{tiny}

\begin{figure}[!htpb]
\scalebox{1.0}{\includegraphics{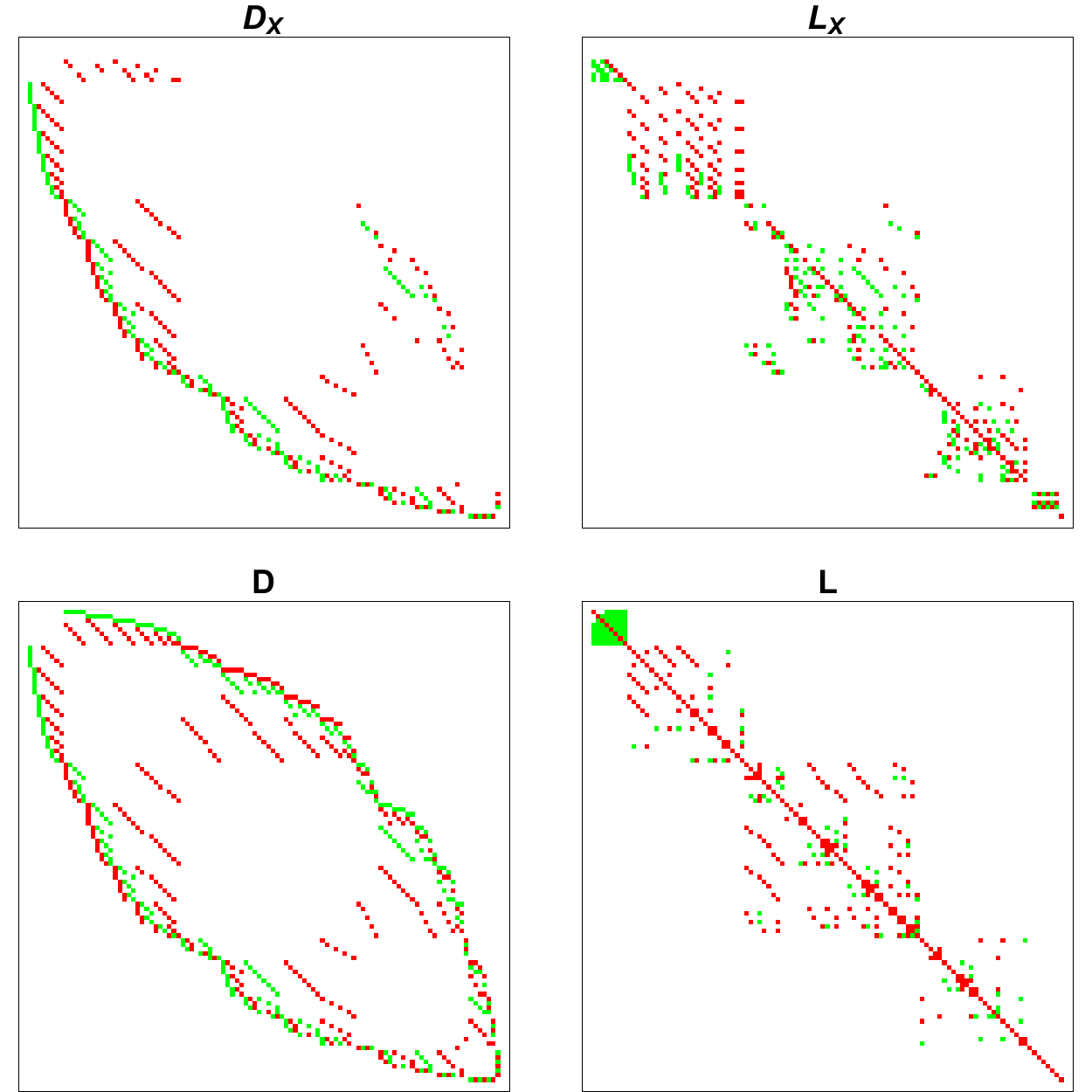}}
\label{example}
\caption{
The matrices $D_X,L_X,D,L$ in the case of a random complex. 
This was produced with the code above. 
}
\end{figure}

\paragraph{}
And finally, here is the self-contained procedure which does the isospectral deformation of 
the exterior derivative by deforming $D_X'=[B_X,D_X]$. In the case
$i_X=d^*$, this is the standard {\bf Lax isospectral deformation} we have seen 
before \cite{IsospectralDirac,IsospectralDirac2}. 
In an other extreme case, when $X=0$, then $d$ is not deformed at all. 
We still have the inflationary initial decay of $d$ typical for that
type of integrable dynamical system. The decay of $d$ means by
the Connes formula that there is an expansion of space because if 
the derivative operator become small, then the distances grow. 

\begin{tiny}
\lstset{language=Mathematica} \lstset{frameround=fttt}
\begin{lstlisting}[frame=single]
Generate[A_]:=Delete[Union[Sort[Flatten[Map[Subsets,A],1]]],1]
R[n_,m_]:=Module[{A={},X=Range[n],k},Do[k:=1+Random[Integer,n-1];
  A=Append[A,Union[RandomChoice[X,k]]],{m}];Generate[A]];G=Sort[R[5,8]];
n=Length[G]; fv=Delete[BinCounts[Map[Length,G]],1];
cn=Length[fv];br={0};Do[br=Append[br,Last[br]+fv[[k]]],{k,cn}];

Orient[a_,b_]:=Module[{z,c,k=Length[a],l=Length[b]},
  If[SubsetQ[a,b] && (k==l+1),z=Complement[a,b][[1]];
  c=Prepend[b,z]; Signature[a]*Signature[c],0]];
d=Table[0,{n},{n}];d=Table[Orient[G[[i]],G[[j]]],{i,n},{j,n}];
dt=Transpose[d]; DD=d+dt; LL=DD.DD; 

UseIntegers=False;e={}; Do[If[Length[G[[k]]]==2,e=Append[e,k]],{k,n}];
BuildField[P_]:=Module[{X,ee,iX=Table[0,{n},{n}]},
X=Table[If[UseIntegers,Random[Integer,1],Random[]],{l,Length[e]}];
Do[ee=G[[e[[l]]]]; Do[If[SubsetQ[G[[k]],ee],
  m=Position[G,Sort[Complement[G[[k]],Delete[ee,2]]]][[1,1]];
  iX[[m,k]]= If[MemberQ[P,Length[G[[m]]]],X[[l]],0]*
             Orient[G[[k]],G[[m]]]],{k,n}],{l,Length[e]}]; iX];
iX=BuildField[{1,3,5,7,9}]; iY=BuildField[{1,3,5,7,9}];
DX=iX+d; LX=DX.DX; DY=iY+d; LY=DY.DY; iZ=LX.iY-iY.LX; DZ=iZ+d;LZ=DZ.DZ;

T[A_]:=Module[{n=Length[A]},Table[If[i<=j,0,A[[i,j]]],{i,n},{j,n}]];
UT[{DD_,br_}]:=Module[{D1=T[DD]},     (* Lower triangular block  *)
Do[Do[Do[D1[[br[[k]]+i,br[[k]]+j]]=0,{i,br[[k+1]]-br[[k]]}],
 {j,br[[k+1]]-br[[k]]}],{k,Length[br]-1}];D1];
RuKu[f_,x_,s_]:=Module[{a,b,c,u,v,w,q},u=s*f[x]; (* Runge Kutta  *)
a=x+u/2;v=s*f[a];b=x+v/2;w=s*f[b];c=x+w; q=s*f[c];x+(u+2v+2w+q)/6];

DD=DX; d0=UT[{DD,br}]; e0=Conjugate[Transpose[d0]];
M=1000; delta=2/M; u={};  (* Deformation with Runge Kutta      *)
Do[d=UT[{DD,br}];e=Conjugate[Transpose[d]];
BB=d-e; CC=d+e; MM=CC.CC; b=DD-CC; VV=b.b;
B=BB+1.0*I*b; f[x_]:=B.x-x.B; DD=RuKu[f,1.0 DD,delta];
u=Append[u,Total[Abs[Flatten[Chop[d]]]]],{m,M}];
DDX=DD; LLX=DDX.DDX;

{Total[Abs[Flatten[Chop[d.d]]]], Total[Abs[Flatten[Chop[e.e]]]]}

F[x_]:=If[x==0,0,-Log[Abs[x]]];  (* Plot the size of d      *)
v=M*Table[F[u[[k+1]]]-F[u[[k]]],{k,Length[u]-1}];
ListPlot[v,PlotRange->All]
\end{lstlisting} \end{tiny}

\begin{figure}[!htpb]
\scalebox{1.0}{\includegraphics{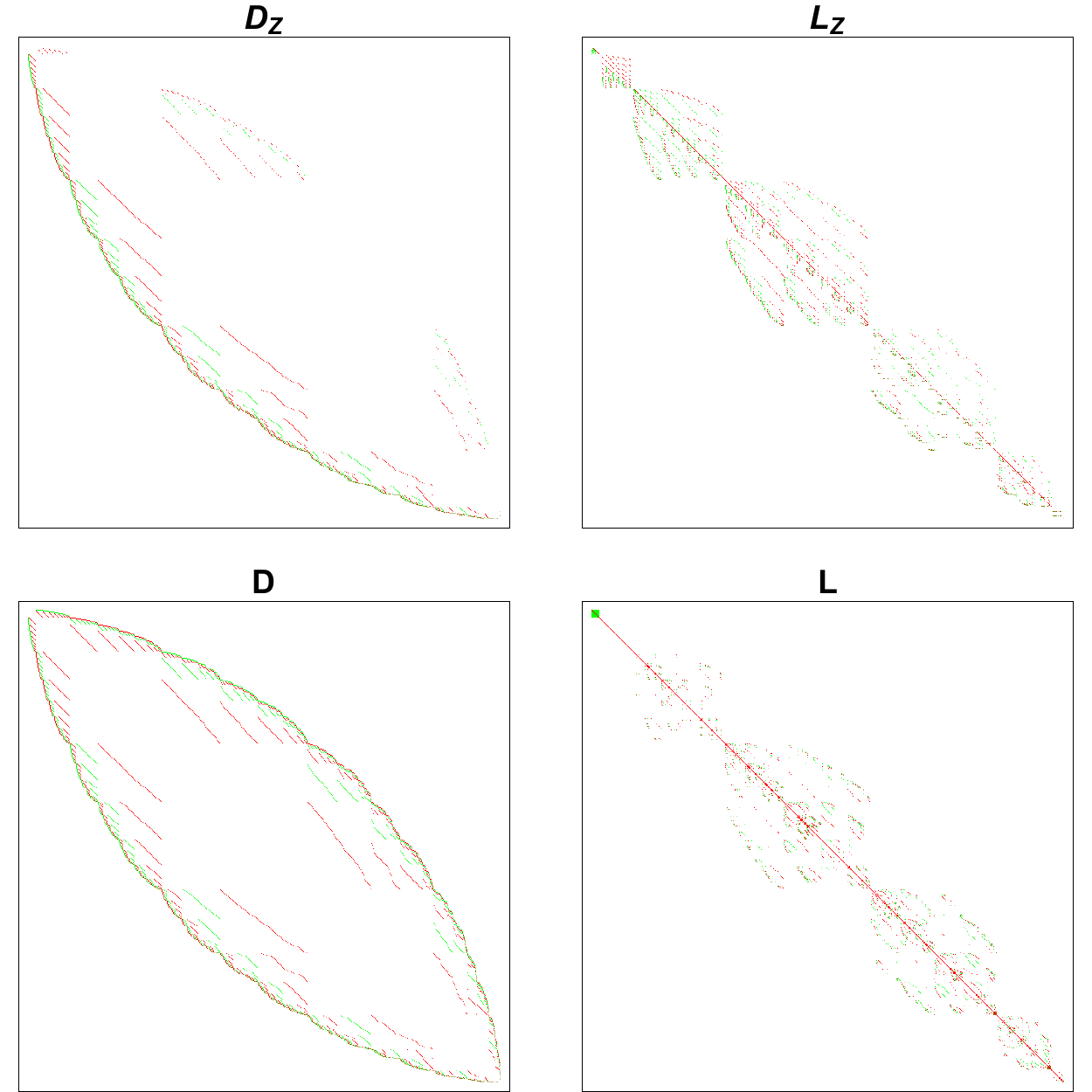}}
\label{example}
\caption{
The matrices $D_Z,L_Z,D,L$ in the case of a random complex
with $f$-vector $(10, 45, 120, 192, 165, 73, 15, 1)$. 
Also this was produced with the code above, where $Z=[X,Y]$
is the commutator of two random vector fields. 
}
\end{figure}

\section{Questions}

\paragraph{}
The operators $D_X$ and $L_X$ are not symmetric in general so that complex eigenvalues can appear.
In that case, the solution $\psi(t) = e^{i D_X t} \psi(0)$ can grow exponentially. 
$D_X$ still often has real eigenvalues, leading to quasi-periodic solutions as the orbits $e^{i D_X t} \psi(0)$ form 
a subgroup of a finite dimensional torus, if the graph is finite.  Actually, if $i_X(k,l) \neq 0$ only for
one $l$, then we implement a deterministic vector field. In that case the eigenvalues of $L_X$ often
non-negative integers taking values in $\{0,1, \dots, {\rm dim}(G)+1 \}$. We would like to understand
the spectrum. 

\paragraph{}
For $D=i_X + i_X^* + d + d^*$ we have $D^2 = L_X + L_X^* + L$. Now, if we average that over all 
possible vector fields using a measure which is homogeneous, we expect the $L_X$ to
average out and get the wave equation governed by $L$. Can one make this more precise and see
the wave equation $f_{tt} = - L$ as an average of deterministic flows $f_{tt} = - L_X$? 

\paragraph{}
We often integer eigenvalues of $L_X$ if $i_{X}$ has integer values. In small dimensional examples,
we can compute general formulas for the eigenvalues but integer eigenvalues also often appear for
large random simplicial complexes. Under which conditions does $L_X$ have integer eigenvalues? 

\paragraph{}
The eigenvalues of $L_X$ are most of the time real if the entries of $I_X$ are non-negative 
multiplies of $d^*$. They can become imaginary in general, if the signs are changed. 
Can we find conditions which assures a real spectrum? 

\bibliographystyle{plain}

\end{document}